\documentclass[12pt]{amsart}
\usepackage{amssymb,amsmath,mathtools}
\usepackage{color}
\usepackage[shortlabels]{enumitem}
\usepackage[backref=page]{hyperref}
\usepackage{cleveref}
\usepackage{url}
\renewcommand*{\backref}[1]{}
\renewcommand*{\backrefalt}[4]{%
	\ifcase #1 (Not cited).%
	\or        (Cited on page~#2).%
	\else      (Cited on pages~#2).%
	\fi}

\theoremstyle{plain}

\newtheorem{theorem}{Theorem}[section]
\newtheorem{theoremintro}{Theorem}

\newtheorem{lemma}[theorem]{Lemma}
\newtheorem{proposition}[theorem]{Proposition}

\newtheorem{corollary}[theorem]{Corollary}

\theoremstyle{definition}
\newtheorem{definition}[theorem]{Definition}
\newtheorem{remark}[theorem]{Remark}

\theoremstyle{remark}
\newtheorem{example}[theorem]{Example}

\numberwithin{equation}{section}
\numberwithin{figure}{section}
\numberwithin{table}{figure}

\usepackage{todonotes}

\newcommand*{\rom}[1]{\expandafter\@slowromancap\romannumeral #1@}

\newcommand{\pt}[1]{\left({#1}\right)}
\newcommand{\pq}[1]{\left[{#1}\right]}

\newcommand{\ps}[2]{\left\langle{#1},{#2}\right\rangle}
\newcommand{\rest}[2]{\left.{#1}\right|_{#2}}

\newcommand{\pg}[1]{\left\{{#1}\right\}}
\newcommand{\abs}[1]{\left|{#1}\right|}

\newcommand{\alb}[1]{\,\alpha^{\overline{#1}}}
\newcommand{\al}[2]{\,\alpha^{#1\bar{#2}}\,}
\newcommand{\aldue}[3]{\,\alpha^{#1\bar{#2}\bar{#3}}\,}

\newcommand{\Z}{\mathbb{Z}}

\newcommand{\R}{\mathbb{R}}

\newcommand{\C}{\mathbb{C}}

\newcommand{\hr}{Hodge-Riemann }
\newcommand{\hrt}{Hodge-Riemann type}

\newcommand{\del}{\partial}
\newcommand{\delbar}{\bar\partial}

\title{On the existence of  balanced metrics  of Hodge-Riemann type}
\author{Anna Fino}
\address[Anna Fino]{Dipartimento di Matematica ``G. Peano'', Universit\`{a} degli studi di Torino \\
Via Carlo Alberto 10\\
10123 Torino, Italy\\
\& Department of Mathematics and Statistics, Florida International University\\
Miami, FL 33199, United States}
\email{annamaria.fino@unito.it, afino@fiu.edu}

\author{Asia Mainenti}
\address[Asia Mainenti]{“Simion Stoilow” Institute of Mathematics of the Romanian Academy\\
 Calea Grivi\c{t}ei 21, Bucharest, Romania}
 \email{asia.mainenti@imar.ro}

\keywords{balanced, Hodge-Riemann, nilmanifold, solvmanifold}
\subjclass[2020]{Primary: 
53C55; %Global differential geometry of Hermitian and Kählerian structures
53B35, %Local differential geometry of Hermitian and Kählerian structures
Secondary: 
22E25, %nilpotent and solvable lie groups
22E60 % Lie algebras of Lie groups
}

\date{\today}

\begin{document}

\begin{abstract}
In the paper we study the existence of balanced metrics of Hodge-Riemann type on non-K\"ahler complex manifolds.
We  first find some  general obstructions, for instance that a Hodge-Riemann balanced manifold of complex dimension $n$ has to be $(n - 2)$-K\"ahler.
Then, we focus on the case of compact quotients of Lie groups by lattices, endowed with an invariant  complex structure.
In particular, we  prove  non existence results on non-K\"ahler complex parallelizable manifolds and some classes of solvmanifolds, and we show that the only nilmanifolds admitting invariant structures of this type are tori.
Finally, we construct the first non-K\"ahler example of a Hodge-Riemann balanced structure, on a non-compact complex manifold obtained as the product of the Iwasawa manifold by $\C$.
\end{abstract}

\maketitle

\section{Introduction}

The non-abelian Hodge correspondence for K\"ahler manifolds  was  recently extended  to a larger class of hermitian metrics on complex manifolds,  called balanced of Hodge-Riemann type.
In particular, in  \cite{cw}, the authors prove that under this assumption, 
on a compact complex manifold, there is a $1$-$1$ correspondence between  semisimple flat bundles and  isomorphism classes of $F$-polystable Higgs bundles satisfying further assumptions, where $F$ is the fundamental form of a Hodge-Riemann balanced metric.

We recall that a  Hermitian  metric $g$ on a complex manifold $(X, J)$ is called \textit{balanced} if its fundamental form $F$ satisfies the condition $d
F^{n -1} =0$, where $n$ is the complex dimension of $X$. A  balanced metric $g$ is  said of \textit{Hodge-Riemann type}  {\cite[Def. 2.7]{cw}} if $F^{n -1}$ 
 can be written as 
        $$
        \frac{F^{n-1}}{(n-1)!}=\omega\wedge\Omega,
        $$
with   $\omega$  a positive definite  $(1,1)$-form  and  $\Omega$  a  closed $(n-2,n-2)$-form  such that $\Omega$ is of Hodge Riemann type with respect to $\omega$ for degrees $(p,q)$,  with $p + q =2$.
This is a generalization of the Hodge-Riemann bilinear relations, which are known to hold on K\"ahler manifolds \cite{GriffithHarris},  and  have been extended to more general results (see, for instance, \cite{timorin,dinh,RT}).

As mentioned above, what was proved in \cite{cw} gives restrictions to non-K\"ahler compact complex manifolds admitting 
Hodge-Riemann balanced structures.
However, the only known examples of Hodge-Riemann balanced metrics, even non-K\"ahler ones, arise in the K\"ahler setting.
For this reason, the main topic that interests this paper is the study of such structures on non-K\"ahler manifolds.

In  the first part of the paper we show that the existence of a balanced metric of Hodge-Riemann type on a complex manifold $(X,J)$ of complex dimension $n$  forces    the  manifold to be  $(n-2)$-K\"ahler. 
For $1\le p\le n$, a $p$-K\"ahler structure on $(X,J)$ is a closed transverse $(p,p)$-form \cite{AA}, coinciding with a K\"ahler metric, for $p=1$,  the $(n-1)$ power of balanced metric, for $p=n-1$ and any volume form, for $p=n$.
More precisely,  we prove  the following.

\smallskip

\begin{theoremintro}\label{thmpkahler}   A Hodge-Riemann balanced manifold $X$ of complex dimension $n$  admits a  $(n-2)$-K\"ahler  structure.
In particular,    if  $n =3$,   $X$  has to be  K\"ahler. \end{theoremintro} 

\smallskip

The study of  $(n -2)$-K\"ahler structures allows us to find more obstructions  to the existence of  Hodge-Riemann balanced  structures on  non-K\"ahler manifolds.

Using that $(X, J)$ has to be $(n-2)$-K\"ahler,   by  our previous result \cite[Thm. 3.7]{fm},  we  can exclude the existence of Hodge-Riemann balanced metrics  on nilmanifolds  of complex dimension $n =4$, endowed with an invariant complex structure, namely a complex structure induced by one at the Lie algebra level.
However, if  $n >4$, no general result is known about the existence of $(n -2)$-K\"ahler structures on nilmanifolds, and the only known examples of such structures are on odd-dimensional complex parallelizable nilmanifolds \cite{AB}.

On the other hand, balanced nilmanifolds have been of wide interest. 
In complex dimension $3$ these have been classified \cite{Ugarte}, and further studied in \cite{uv}.
In complex dimension higher than $3$, only a partial classification is known, in the case of $1$-dimensional center and complex dimension $4$, in \cite{luvFrolicher}.

Motivated by this, we consider the  following question.

\smallskip

\noindent
\textbf{Question 1.}
Let $ (X, J)$ be a nilmanifold endowed with an invariant complex structure.
Does $(X,J)$ admit an invariant balanced metric of Hodge-Riemann type?

\smallskip

Note that if the balanced Hodge-Riemann structure is invariant, the problem can be stated in terms of nilpotent Lie algebras.
As a consequence, in Theorem \ref{csnilpNO}, we prove that Question 1 can only be true if $X$ is a complex torus, for every complex dimension $n\ge 3$.
The problem remains open in the case of non invariant balanced structures, even when the complex structure is invariant.

Question 1 can be  extended to the more general  case of  a compact quotient $\Gamma \backslash G$  of a simply connected Lie group $G$  by a  discrete subgroup $\Gamma$,  endowed with an invariant complex structure $J$. 
We consider in particular the cases where $J$ is either bi-invariant or  abelian, proving in both instances that   if $\Gamma \backslash G$  admits an invariant balanced metric of Hodge-Riemann type, then  $(\Gamma \backslash G, J)$ has to be K\"ahler. 
Furthermore, the argument used in the case of bi-invariant complex structures allows to extend the result to non invariant Hodge-Riemann balanced metrics.

In the case when  $\Gamma \backslash G$ is a solvmanifold, i.e. if the Lie algebra  $\mathfrak g$ of $G$  is solvable,  the existence of  balanced structures was extensively studied, for instance   if $G$ is  almost abelian \cite{fp23}, if $\mathfrak g$  has a codimension two  abelian ideal \cite{CaoZheng,GuoZheng}, in the $2$-step solvable case \cite{freibSwann}, in the case of abelian complex structures \cite{AndradaV} and on almost nilpotent Lie algebras \cite{fp22,FParxiv}.

By using \cite[Thm. 4.2]{fm},  we   can  show that  if an   almost  abelian solvmanifold  with an invariant complex structure is Hodge-Riemann  balanced, then it  has to be  K\"ahler. 
In Section \ref{secSolv}, we extend this result to solvmanifolds  with an invariant complex structure  $J$ such that the associated solvable Lie algebra  has an abelian $J$-invariant ideal of  real codimension  $2$.

To sum up, for a compact quotient $X$  of a Lie group $G$ by a lattice endowed with an invariant complex structure  $J$,  we prove the following.

\smallskip
  \begin{theoremintro} \label{ThmIntro}
   Let  $X$ be  the compact quotient of a Lie group $G$ by a lattice endowed with an invariant complex structure $J$.
   The existence of a  Hodge-Riemann balanced  structure $(F, \omega, \Omega) $ implies K\"ahlerianity  in the following cases:
    \begin{enumerate}[label=(\roman*)]    
        \item $G$ is a complex Lie group, i.e. $X$ is a complex parallelizable manifold;
        \item  $X$ is an almost abelian solvmanifold;
        \item\label{abideal} the Lie algebra $\mathfrak g$ of $G$ has a $J$-invariant abelian ideal of  real codimension two.
    \end{enumerate}
    If  we assume in addition that  the balanced metric is invariant, $(X,J)$ still has to be K\"ahler in the following cases:
    \begin{enumerate}[label=(\roman*),resume]
        \item $X$ is a nilmanifold;
        \item the complex structure $J$ is abelian;
        \item the commutator ideal $[\mathfrak g,\mathfrak g]$ is not $J$-invariant.
    \end{enumerate} \end{theoremintro}

\smallskip

In contrast with the many obstructions arising in the compact case, we will conclude giving the first non-K\"ahler example  of Hodge-Riemann balanced structure, on  a non-compact  manifold.
More precisely, we prove the following.
\smallskip
\begin{theoremintro}\label{thmintroIC}
    The product $M={\mathcal{I}\times \C}$ of the Iwasawa manifold $\mathcal{I}$ by $\C$ admits a non-invariant balanced metric of Hodge-Riemann type.
\end{theoremintro}
The organization of the paper is as follows. 
In Section \ref{secHR} we go into the details of the definition of balanced structures of Hodge-Riemann type and fix some notations and remarks that will be useful in what follows.

Section \ref{secpK} is dedicated to show how Hodge-Riemann balanced manifolds can be seen as a particular instance of $p$-K\"ahler manifolds, for $p=n-2$, and consequently inherit obstructions from the theory of the latter.
A straightforward consequence is Theorem \ref{thmpkahler}.
The main result of this section is Lemma \ref{lemmaobst}, an obstruction result in terms of harmonic, i.e. $\del\delbar$-closed forms, that will play the main role in the proof  of Theorem \ref{ThmIntro}. %, in the last section.

In Section \ref{secSolv} we specialize what was noted and proved in the previous sections, to the case of Lie groups.
In particular, we study compact quotients of Lie groups, endowed with invariant complex structures.
In this setting, at the Lie algebra level, we can restate some of the previous results in more simple terms, using that for invariant balanced metrics, exact forms are always primitive.
This will allow us to prove the last three items of Theorem \ref{ThmIntro}.
Later on, as mentioned above, we will prove an analogue of \cite[Thm. 4.2]{fm}, for Lie algebras admitting a $J$-invariant abelian ideal of codimension $2$.
Using this, we will prove \ref{abideal} of  Theorem \ref{ThmIntro}.

To conclude, in Section \ref{section.noncompact}, we prove Theorem \ref{thmintroIC}, by an explicit construction of a Hodge-Riemann balanced structure.
We produce a $(n-2)$-K\"ahler form $\Omega$ that also satisfies the Hodge-Riemann bilinear relations, with respect to a positive definite $(1,1)$-form $\omega_0$.
Both the forms $\omega_0$ and $\Omega_0$ will be invariant, along the Iwasawa manifold, but not along $\C$.
Also, the wedge $\omega_0\wedge\Omega$ will produce a balanced metric, that will not be invariant, accordingly to Theorem \ref{ThmIntro}.

\section{Balanced metrics of Hodge-Riemann type}\label{secHR}

On a complex vector space $V$ of dimension $n$, denote by $\Lambda^{p,q}: =\Lambda^{p,q}V^*$ the space of $(p, q)$-forms over $V$.
We will now review some positivity notions for forms on $V$, which we will then apply to the case of differential forms on complex manifolds $(X,J)$, with $V=T_xX$, for $x\in X$.
In this setting, we will identify Hermitian metrics on  $(X,J)$ with the associated strongly positive $(1, 1)$ forms.

\begin{definition}\label{poscondHRT}
    A $(k,k)$-form $\Omega$  on $V$  is \textit{strongly positive} if %it can be written as
    $\Omega=i^{{k}^2}\sum_j\psi_j\wedge\bar{\psi}_j$, with $\psi_j\in\Lambda^{k,0}$ decomposable generators of $\Lambda^{k,0}$.

    A $(k,k)$-form $\Omega$  on $V$ is \textit{positive definite} if, for all $\eta\in\Lambda^{n-k,0}$, the $(n,n)$-form 
    $$
    i^{\pt{n-k}^2}\,\Omega\wedge\eta\wedge\bar\eta
    $$
    is a positive multiple of a fixed volume form. 
    A $(k,k)$-form $\Omega$ on a complex vector space $V$ is \textit{transverse} if $i^{\pt{n-k}^2}\,\Omega\wedge\psi\wedge\bar\psi$ is a positive volume form, for all $\psi\in\Lambda^{n-k,0}$ decomposable.
\end{definition}

The wedge product of strongly positive forms is still a strongly positive form, whereas the same is false in general for positive definite forms \cite[Prop. III.1.11]{demailly}.
In fact, the most well-known examples of strongly positive forms are powers of Hermitian metrics.
It is evident that strongly positive forms are positive definite, and positive definiteness implies transversality.
The latter is in some sense dual to strong positivity, and  transverse $(k,k)$-forms are characterized from the property of being volume forms, when restricted to complex subspaces of $V$ of complex dimension $k$.
The following allows for a better understanding of positive definiteness.

\begin{remark}\label{rmkPosDef}
For every $(p,q)$ with $p+q\eqqcolon k<n$, one can associate to each real form $\Omega\in\Lambda^{n-k,n-k}$ a Hermitian form $Q=Q_{\Omega}^{p,q}$ on $\Lambda^{p,q}$, defined by
$$
Q(\alpha,\beta)\coloneqq i^{p-q}\pt{-1}^\frac{k\pt{k-1}}{2}\star\pt{\alpha\wedge\bar{\beta}\wedge\Omega},
$$
where $\star$ is the Hodge star operator associated to a Hermitian metric $F$. 
When $q=0$, we can see that $Q_\Omega^{k,0}$ is positive definite if and only if $\Omega$ is positive definite, in the sense of Definition \ref{poscondHRT}, with respect to the volume form $\frac{F^n}{n!}$, or equivalently to a volume form obtained by  any other Hermitian metric.
\end{remark}

\begin{definition}\label{Hodge-Riemann form}
    A real $(n-p-q,n-p-q)$ form $\Omega$ on $V$ is called a \emph{{Hodge-Riemann form}} with respect to a strongly positive $(1, 1)$ form $\omega$,  for degree $(p,q)$ if
    \begin{itemize}
        \item $Q$ is positive definite on the space of primitive forms 
$$
{{P^{p,q}}}\coloneqq\pg{\alpha\in\Lambda^{p,q}\colon\alpha\wedge\omega\wedge\Omega=0},
$$
        \item $\Lambda^{p,q}$ decomposes $Q$-orthogonally in 
        $$
        \Lambda^{p,q}=\omega\wedge\Lambda^{p-1,q-1}\oplus P^{p,q}.
        $$
    \end{itemize}
\end{definition}

By construction, the second condition is intended to be always satisfied, when $pq=0$.
Moreover, in the remaining cases, a simple computation shows that if $\omega\wedge\Lambda^{p-1,q-1}$ is in direct sum with $P^{p,q}$, then the two spaces are $Q$-orthogonal, so this condition is guaranteed by 
$\omega^2\wedge\Omega\neq0$, .

An example of {Hodge-Riemann form} with respect to a strongly positive $(1, 1)$-form $\omega$, for any degree $(p,q)$ is $\omega^{n-p-q}$ (cf. \cite{voisin}).
More in general, if $\omega_0,\dots,\omega_{n-p-q}$ are strongly positive $(1, 1)$-forms, then $\omega_1\wedge\dots\wedge\omega_{n-p-q}$ is a Hodge-Riemann form with respect to $\omega_0$ (cf. \cite{timorin}).

In the setting of a complex manifold $X$ of complex dimension $n$,  we can give the following definition.

\begin{definition}[{\cite[Def. 2.7]{cw}}]\label{HRtype}
A Hermitian  balanced metric $F$ on a complex manifold $(X,J)$ is said to
be of \textit{Hodge-Riemann type} if the following conditions hold:
\begin{enumerate}[label=(\alph*)]
    \item There exist a Hermitian metric $\omega$ and a real $(n - 2, n - 2)$ form $\Omega$ such that 
        $$
        \frac{F^{n-1}}{(n-1)!}=\omega\wedge\Omega.
        $$
    \item\label{HRpq} At every point $\Omega$ is a Hodge-Riemann form with respect to $\omega$ for all the degrees $(p,q)$ such that $p+q=2$;
    \item $\Omega$ and $\omega\wedge\Omega$ are closed (namely, $\Omega$ is closed and $F$ is balanced).
\end{enumerate}
We will refer to the triple $(F,\omega,\Omega)$ as to a \textit{Hodge-Riemann balanced structure} on $(X,J)$.
\end{definition} 

Note that from the first condition it follows that a $(1,1)$-form $\alpha$ is primitive in the sense defined above if and only if it is primitive in the usual sense, i.e. $F^{n-1}\wedge\alpha=0$, with $F$ balanced.

\begin{example} 
    Let $\omega_0,\dots,\omega_{n-2}\in\Lambda^{1,1}$ be strongly positive forms.
    As mentioned above, $\Omega=\omega_1\wedge\cdots\wedge\omega_{n-2}$ is a strongly positive $(n - 2, n - 2)$-form.
    Moreover (see for example \cite{michelsohn}), there exists a strongly positive $(1,1)$-form $F$ such that
    $$
    \frac{F^{n-1}}{(n-1)!}=
    \omega_0\wedge\omega_1\wedge\cdots\wedge\omega_{n-2},
    $$
    so that if $\omega_0,\dots,\omega_{n-2}$ are closed, $F$ is  Hodge-Riemann balanced.
\end{example}

\section{Relation with $(n -2)$-K\"ahler structures and obstructions}\label{secpK} 

A reformulation of \Cref{HRtype}, focused on one of the three forms defining a Hodge-Riemann balanced structure (\Cref{HRequiv}), emphasizes the relation of the latter with $p$-K\"ahler structures (\Cref{theoremn-2}). This allows for general obstruction results to the existence of Hodge-Riemann balanced structures, such as \Cref{obstructcpd,obstructhrt}, culminating in \Cref{lemmaobst}, a more specialized, but more explicit, criterion for non-existence of such structures on balanced manifolds.

\medskip
We begin by restating Definition \ref{HRtype} in an equivalent way, to emphasize that the Hodge-Riemann conditions are actually positivity conditions on $\Omega$.

\begin{proposition}\label{HRequiv}
A Hermitian balanced metric $F$ on a complex manifold $(X,J)$ is of {Hodge-Riemann type} if and only if
        $$
        \frac{F^{n-1}}{(n-1)!}=\omega\wedge\Omega,
        $$
for some Hermitian metric $\omega$ and a closed $\Omega\in\Lambda^{n-2,n-2}$, such that
\begin{enumerate}[label=(\roman*)]
    \item\label{hrequiv1} $\Omega$ is positive definite;
    \item\label{hrequiv2}  $Q_\Omega$ is positive definite on the space of primitive $(1,1)$-forms $P^{1,1}$.
\end{enumerate}
\end{proposition}

\begin{proof}
{Comparing the statement with \Cref{HRtype}, it is sufficient to} show that condition \ref{HRpq} in Definition \ref{HRtype} is equivalent to \ref{hrequiv1}, for $p=0$, or $q=0$, and to \ref{hrequiv2}, for $p=q=1$.

{Indeed, for $p=0$, respectively $q=0$, Definition \ref{HRtype}.\ref{HRpq} states that $Q_\Omega$ is positive definite on the space $\Lambda^{0,2}$, respectively  $\Lambda^{2,0}$.
    The equivalence between this condition and \ref{hrequiv1} is therefore a straightforward consequence of Remark \ref{rmkPosDef}, where the definition of $Q_\Omega$ is given.
    \\
    On the other hand, for the case $p=q=1$, condition \ref{HRpq} in Definition \ref{HRtype}}
is equivalent to condition \ref{hrequiv2}, together with the $Q_\Omega$-orthogonal splitting $\Lambda^{p,q}=\omega\wedge\Lambda^{p-1,q-1}\oplus P^{p,q}$.
    As mentioned above, {this orthogonality} is equivalent to require
    $$
    0\neq\omega^2\wedge\Omega=\omega\wedge\frac{F^{n-1}}{(n-1)!}.
    $$
    By construction, $\omega$ and $F^{n-1}$ are both strongly positive and we already noted that the wedge of strongly positive forms is a strongly positive form, so in particular is non zero, allowing us to conclude.
\end{proof}

We will see in what follows that many obstructions to the existence of non-K\"ahler Hodge-Riemann balanced manifolds come from the study of closed positive definite forms.
This is a particular instance of a more broad class of structures, $p$-K\"ahler structures \cite{AA}, defined as closed transverse differential forms.

\begin{definition} 
Let $(X, J)$  be a complex manifold  of complex dimension $n$ and $1 \leq p \leq n$.  A \textit{$p$-K\"ahler structure}  on $X$ is   given by  a $d$-closed, pointwise transverse, $(p,p)$-form.
\end{definition}

Clearly, a complex manifold is  always $n$-K\"ahler.
For $p=1$ (respectively $p = n -1$)  we obtain the well-known case of K\"ahler (respectively balanced) metrics. Note that  for $1 < p < n-1,$  $p$-K\"ahler structures  have no metric meaning \cite[Proposition 2.1]{AB}.

\begin{theorem}\label{theoremn-2}
    A Hodge-Riemann balanced manifold of complex dimension $n$ is  $(n-2)$-K\"ahler.
    In particular, Hodge-Riemann balanced manifolds of complex dimension $3$ are K\"ahler.
\end{theorem}

\begin{proof}
     If $(F,\omega,\Omega)$ is a Hodge-Riemann balanced structure, then  {by \Cref{HRequiv}.\ref{hrequiv1}} $\Omega$ is a closed positive definite form, so in particular it is $p$-K\"ahler, with $p=n-2$.
\end{proof}

A first obstruction to the existence of $p$-K\"ahler structures is the following.
\begin{proposition}[{\cite[Prop. 3.4]{hmt21}}]\label{obstructpK}
    Let $(X,J)$ be a compact complex manifold of dimension $n$.
    Suppose there exists a form $\beta\in\Lambda^{2n-2p-1}$ such that   
    \begin{equation} 
        0\neq\pt{d\beta}^{n-p,n-p}=\sum_jc_j\,\psi_j\wedge\overline{\psi_j},
    \end{equation}
    where $c_j$ have the same sign and $\psi_j$ are simple $(n-p,0)$-forms.
    Then $(X,J)$ does not admit {$p$-K\"ahler} forms.
\end{proposition}

This property can be generalized to closed positive definite forms.

\begin{proposition}\label{obstructcpd}
    Let $(X,J)$ be a compact complex manifold of complex dimension $n$.
    Suppose there exists a form $\gamma\in\Lambda^{2n-2p-1}$ such that   \begin{equation} 
        0\neq\pt{d\gamma}^{n-p,n-p}=
        \sum_jc_j\,\eta_j\wedge\bar{\eta}_j,
    \end{equation}
    where the $c_j$ have the same sign and $\eta_j\in\Lambda^{n-p,0}$.
    Then $(X,J)$ does not admit closed positive definite $(p,p)$-forms.
\end{proposition}

\begin{proof}
     {We argue by contradiction.}
    Suppose there exists a closed positive definite $(p,p)$-form $\Omega$.
    Since $X$ is compact, we have 
    $$
    0=\int_Xi^{\pt{n-p}^2}\,d\pt{\Omega\wedge\gamma}
    =\int_Xi^{\pt{n-p}^2}\,{\Omega\wedge d\gamma}
    =\sum_jc_j\int_Xi^{\pt{n-p}^2}\,{\Omega\wedge\eta_j\wedge\bar{\eta}_j}.
    $$
    If we assume all $c_j$ to be positive, the positive definiteness of $\Omega$ forces  {the above sum} to be positive, giving a contradiction.
    The case where all $c_j$ are negative is analogous.
\end{proof}

We can specialize this for $p=n-2$, to use it in the setting of balanced  structures of Hodge-Riemann type.
In fact, we will always use a particular case of this result.

\begin{corollary}\label{corpd}
    If a compact complex manifold  $(X,J)$ admits a  $3$-form $\gamma$ such that $\pt{d\gamma}^{2,2}=\eta\wedge\bar{\eta}$ for some $0\neq\eta\in\Lambda^{2,0}$, then  $(X,J)$ is not \hr balanced.
\end{corollary}

A similar argument holds, using the positivity condition for \hr forms on the space of primitive $(1,1)$-forms.

\begin{proposition}\label{obstructhrt}
    Let $(X,J)$ be a compact complex manifold of complex dimension $n$ admitting  a balanced metric $F$.
    Suppose there exists a form $\gamma\in\Lambda^{3}$ such that   \begin{equation} 
        0\neq\pt{d\gamma}^{2,2}=
        \sum_jc_j\,\mu_j\wedge\bar{\mu}_j,
    \end{equation}
    where the $c_j$ have the same sign and $\mu_j$ are primitive $(1,1)$-forms.
    Then $F$ is not of \hrt. 
\end{proposition}

\begin{proof}
     {We argue by contradiction, assuming there are some $\omega,\Omega$ such that}
$(F,\omega,\Omega)$ is a Hodge-Riemann balanced structure.
    By Proposition \ref{HRequiv}\ref{hrequiv1}, $Q_\Omega$ is a positive definite  Hermitian form on the space of $F$-primitive $(1,1)$-forms.
    We conclude, as in Proposition \ref{obstructcpd}, integrating 
    $
    d\pt{\Omega\wedge\gamma}=\Omega\wedge d\gamma =\sum_jc_j\,\Omega\wedge \mu_j\wedge\bar{\mu}_j
    $ over $X$.
\end{proof}

A specific instance of this obstruction can be seen as the existence of  $\del$-exact, or $\delbar$-exact, closed forms.

\begin{lemma}\label{lemmaobst}
    Let  $F$ be a balanced metric on  a compact  complex manifold $(X, J)$.
    If $(X,J)$ admits a $(1,0)$-form $\alpha$ that  satisfies one of the following conditions:
    \begin{enumerate}[label=(\roman*)]
    \item\label{lobst1} $\del\alpha\neq0$ and $\del\delbar\alpha=0$;
    \item\label{lobst2} $\delbar\alpha\neq0$, $\del\delbar\alpha=0$ and $\bar\partial\alpha$ is primitive,
    \end{enumerate}
     then $F$ is not of Hodge-Riemann type.
\end{lemma}

\begin{proof}
    {If \cref{lobst1} holds},  consider $\gamma_1=\del\alpha\wedge\bar\alpha\in\Lambda^{2,1}$, with 
    $$
    (d\gamma_1)^{2,2}=\delbar\gamma_1=-\delbar\del\alpha\wedge\bar\alpha-\del\alpha\wedge\delbar\bar\alpha=-
    \del\alpha\wedge\delbar\bar\alpha.
    $$
    Using Corollary \ref{corpd}, we have the thesis,  {because $\partial\alpha$ is a $(2,0)$-form, non-zero by assumption}. 
    Case \ref{lobst2} is analogous, as choosing $\gamma_2=\delbar\alpha\wedge\bar\alpha\in\Lambda^{1,2}$, we obtain
    $$
    (d\gamma_2)^{2,2}=\del\gamma_2=-\del\delbar\alpha\wedge\bar\alpha-\delbar\alpha\wedge\del\bar\alpha=-\delbar\alpha\wedge\del\bar\alpha,
    $$
    and we conclude using Proposition \ref{obstructhrt}.
\end{proof}

\section{Results on Lie groups and compact quotients}\label{secSolv}
{In this section, we will specialize the results obtained in the previous section to the setting of Lie groups.
More precisely, we will begin with the case of complex Lie groups and compact complex parallelizable manifolds (\Cref{prop4.1}), proving \Cref{ThmIntro}.(i). 
Afterwards, we will start working on compact quotients of Lie groups by lattices, endowed with invariant complex structures, shifting the discussion at the Lie algebra level.
First of all, we obtain an improved version of \Cref{lemmaobst} for unimodular Lie algebras (\Cref{ddbarnonclosed}), which will be the main tool in most of the following proofs.
In fact, direct consequences of the latter are \Cref{cor4.7}, for abelian complex structures, and \Cref{hrtJg^1}, where the derived Lie algebra is considered, allowing to prove, respectively, items (v) and (vi) in \Cref{ThmIntro}.
Furthermore, we discuss the case of nilpotent Lie algebras in \Cref{csnilpNO}, proving \Cref{ThmIntro}.(iv).
\\
We conclude the section considering Lie algebras admitting a $J$-invariant abelian ideal of real codimension $2$, ultimately proving that in this case $(n-2)$-K\"ahler implies K\"ahler, in \Cref{ThmAbIdeal}, and consequently concluding the proof of \Cref{ThmIntro}.
}

\medskip
Let us now begin with the case of complex Lie groups and compact complex parallelizable manifolds.
 {We recall that these two families are connected by a result by Wang \cite{wang}, stating that a compact complex parallelizable manifold is
biholomorphic to a quotient of a complex Lie group by a discrete subgroup.}
\begin{proposition}\label{prop4.1}
If a complex Lie group $G$ admits a Hodge-Riemann balanced structure $(F, \omega,\Omega)$, then $\Omega$ cannot be left-invariant, unless $G$ is abelian.
    Similarly, the only complex parallelizable manifolds admitting balanced metrics of Hodge-Riemann type are complex tori.
\end{proposition}

\begin{proof}
   The first part is an immediate consequence of \cite[Cor. 2.6(2)]{AlessandriniLieGroups}, stating that a complex Lie group admitting a left-invariant, closed, positive definite $(p,p)$-form, for some $p<n-1$, has to be abelian.
    For the second part, recall that if  $X$ is a complex parallelizable manifold, it has $n$ holomorphic $1$-forms $\omega^j$.
    Unless  $X$ is a complex torus, there exists some $j$ such that $\omega^j$ is non-closed.
    In particular, since $\omega^j$ is holomorphic, $\delbar\omega^j=0$, hence $\omega^j$ falls in case \ref{lobst1} of Lemma \ref{lemmaobst}.
\end{proof}

Notice that a non-compact, non-K\"ahler complex Lie group could still admit a non invariant positive definite  $(n -2)$-K\"ahler form, or even a Hodge-Riemann balanced structure.

Next, we prove a general necessary condition for \hr balanced Lie algebras. 
We recall the following.

\begin{proposition}[{\cite[Prop. 2.5]{fm}}]
    If a $p$-K\"ahler Lie algebra $(\mathfrak g,J, \Omega)$  {of complex dimension $n \geq 3$}  admits a closed $(1,0)$-form $\alpha$, 
   then it has a $p$-K\"ahler $J$-invariant ideal of codimension 2.
\end{proposition}  

As a consequence, for $p=n-2$, we have the following.

\begin{corollary}\label{jinvideal}
    If $(\mathfrak g,J)$ is balanced of \hrt, $\mathfrak g$ admits a $J$-invariant balanced ideal of codimension 2.
\end{corollary}

A well-known result states that for invariant Hermitian metrics on a compact quotient of a Lie group by a lattice,  being balanced is equivalent to the orthogonality of the space of primitive forms with respect to the image of the exterior derivative of differential forms. 
Namely, an invariant Hermitian metric $F$ is balanced if and only if every exact $2$-form is primitive.

\begin{proposition}[\cite{AGS,AndradaV}]\label{andradav}
    Let $(\mathfrak g, J, F)$ be a unimodular Hermitian  Lie algebra of complex dimension $n$.
    Then,  the Hermitian structure $(J,F)$ is balanced if and only if $F^{n-1}\wedge d\alpha=0,$ for every $1$-form $\alpha$.
\end{proposition}

\noindent
In particular, if $(J,F)$ is balanced, then for all $\mu\in\Lambda^{1,0}$, $\delbar\mu$ is a primitive $(1,1)$-form.

\begin{proposition}\label{ddbarnonclosed}
    If a unimodular Lie algebra $\mathfrak g$ admits a  \hr balanced structure $(J, F)$, then every $(1,0)$-form $\alpha$ is either $d$-closed, or both $\del\alpha$ and $\delbar\alpha$ are non-zero.
\end{proposition}

\begin{proof}
     {
    For every balanced structure $(J,F)$ on $\mathfrak g$, if $\alpha$ is a $(1,0)$-form, we have from Proposition \ref{andradav} that $\delbar\alpha$ is primitive with respect to $F$.
    If $d\alpha\neq0$ but $\del\alpha=0$ or $\delbar\alpha=0$, then the assumptions of \Cref{lemmaobst} hold and consequently $F$ cannot be of Hodge Riemann type.}    
\end{proof}

\begin{remark}
    In what follows, we will primarily use this result to find obstructions at the Lie algebra level to the existence of \hr balanced structures.
    Note that if the said obstruction is $\alpha\in\Lambda^{1,0}$, with $\del\alpha=0$, $\delbar\alpha\neq0$, then we will actually be using condition \ref{lobst2} in Lemma \ref{lemmaobst}.
    Now, in the invariant case, namely at the Lie algebra level, $\delbar\alpha$ is always primitive, whereas this does not hold on the underlying manifold, unless the balanced metric we consider is invariant.
    In other words, we are saying that such an invariant $(1,0)$-form will give an obstruction, as in Lemma \ref{lemmaobst}, only if the given balanced metric $F$ is itself invariant.
\end{remark}

As a consequence of this last result, in the unimodular case, we find necessary conditions on the complex structure  on a Lie algebra to admit a compatible  balanced metric of Hodge-Riemann type.

\begin{corollary}\label{cor4.7}
    If $(\mathfrak g,J)$ is a unimodular non abelian Lie algebra endowed with an abelian complex structure, it does not admit balanced metrics of Hodge-Riemann type.
\end{corollary}

\begin{proof}
    Abelian complex structures are characterized by the condition $d\pt{\Lambda^{1,0}}\subset\Lambda^{1,1}.$
    If $\mathfrak g$ is not abelian, there must be some non-closed $(1,0)$-form $\alpha$, and $d\alpha\in\Lambda^{1,1}$, so $\del\alpha=0$.
    We conclude using Proposition \ref{ddbarnonclosed}.
\end{proof}

A more general obstruction, once again consequence of Proposition \ref{ddbarnonclosed}, regards the relation between $J$ and the derived algebra of $\mathfrak g$.

\begin{corollary}\label{hrtJg^1}
    If a unimodular Lie algebra  $\mathfrak g$ has  a Hodge-Riemann balanced structure $(J, F)$, then  the commutator ideal $\mathfrak g^1\coloneqq[\mathfrak g,\mathfrak g]$ has to be  $J$-invariant.
\end{corollary}

\begin{proof}
    The statement is trivially true for an abelian Lie algebra
     {$\mathfrak g.$
    If $\mathfrak g$ is not abelian, we assume that $\mathfrak g^1$ is not $J$-invariant and we apply \Cref{ddbarnonclosed}.}
    
    Suppose that there exists $X\in\mathfrak g^1$ such that $JX\notin\mathfrak g^1$.
    Consider a real $1$-form $\gamma$ on $\mathfrak g$ satisfying $\gamma(X)=1$.
    We have $d\gamma\neq0$ and $dJ\gamma=0$.
    Then, $\alpha\coloneqq\gamma-iJ\gamma$ is a $(1,0)$-form on $\mathfrak g_\C$ and $d\alpha=d\gamma\neq0$ is a real form, so by integrability of $J$, $\alpha$ is $\del$-closed, but not $\delbar$-closed.
    It follows by Proposition \ref{ddbarnonclosed} that $(\mathfrak g,J)$ cannot admit Hodge-Riemann balanced structures.
\end{proof}

We will now discuss the nilpotent case, showing that the only nilpotent Lie algebra admitting \hr balanced structures is the abelian one.

\begin{theorem}\label{csnilpNO}
    Let $(\mathfrak{g},J)$ be a nilpotent Lie algebra of dimension $2n$ endowed with a complex structure.
    If $(\mathfrak{g},J)$ has a Hodge-Riemann balanced metric, then $\mathfrak g$  is abelian.
\end{theorem}

\begin{proof}
    We can see $\mathfrak g^{1,0}$ as a nilpotent complex subalgebra of $\mathfrak g_\C$.
    In terms of the complex structure equations, this means that there exists a basis  $\pg{\alpha^j}_{j=1}^n$ of $\Lambda^{1,0}$ with
    % , such that
    \begin{equation}\label{nilpot}
    \del\alpha^1=0,\quad\del\alpha^j\in\Lambda^2\ps{\alpha^1,\dots}{\alpha^{j-1}},\quad j=2,\dots,n.
    \end{equation}
    Let $t\in\pg{2, \dots,  n}$ be the integer such that $\del\alpha^1=\dots=\del\alpha^{t-1}=0$, but $\del\alpha^t\neq0$.
    If $(\mathfrak{g},J)$ has a balanced Hodge-Riemann metric, by Proposition \ref{ddbarnonclosed} we find that $\alpha^k$ is closed, for all $1\le k<t$.
    This, together with \eqref{nilpot}, implies in particular that $d\,\del\alpha^{t}=0$, meaning that $\alpha^t$ satisfies condition \ref{lobst1} of Lemma \ref{lemmaobst}, allowing us to conclude.
\end{proof}

We end the section with the case of  Lie algebras with an abelian $J$-invariant ideal of codimension $2$.
This includes the case of almost abelian Lie algebras.
In \cite[Thm. 4.2]{fm}, the authors proved that $(n-2)$-K\"ahler almost abelian Lie algebras are actually abelian.
A straightforward consequence is the following.

\begin{proposition}\label{prop4.10}
    An almost abelian solvmanifold, endowed with an invariant complex structure, is \hr balanced if and only if it is the complex torus.
\end{proposition}

Note that we did not require any of the differential forms giving the \hr balanced structure to be invariant.
This is because by symmetrization, \cite[Thm. 4.2]{fm} actually states that almost abelian solvmanifolds with an invariant complex structure cannot be $(n-2)$-K\"ahler, unless K\"ahler.
In what follows, we will see that the same holds in this more general setting.

\begin{proposition}\label{abidn-2k}
    If  $(\mathfrak g,J)$ is a unimodular, $(n-2)$-K\"ahler Lie algebra of complex dimension $n\geq4$, admitting a $J$-invariant abelian ideal $\mathfrak a$ of real codimension $2$, then the complex structure equations can be written as 
    \begin{equation}\label{cseabid}\begin{cases}
    d\alpha^1=0 &\\
    d\alpha^j={ v_j}\alpha^{1\bar 1}-\pt{\bar\lambda_j\alpha^1-{\lambda_j}\alpha^{\bar 1}}\wedge\alpha^j, &   j=2,\dots n,
\end{cases}
\end{equation}
for some $v_j,\lambda_j\in\C$.
\end{proposition}

\begin{proof}
 {The strategy of the proof is the following.
First, we show how to associate to every transverse $(n-2,n-2)$-form $\Omega$ on $\mathfrak g$ a Hermitian metric $\omega$ on $\mathfrak g$.
We will therefore prove that the closure of $\Omega$ forces suitable components of the Chern torsion $T$ of $\omega$ to vanish.
This will be enough to prove the statement, using \cite[eq. (37)]{GuoZheng}, where the components of $T$ are written in terms of the complex structure constants.
}

    Let $\Omega$ be a $(n-2)$-K\"ahler form on $(\mathfrak g,J)$. 
    Then, its restriction $\Omega_\mathfrak a$ to $\mathfrak a$ is a  strongly positive $(n-2,n-2)$-form in complex dimension $n-1$, namely it is the $(n-2)$-th power of  a Hermitian metric $\sigma$ on $\mathfrak a$.
    We can fix a unitary basis $\pg{Z_2,\dots,Z_n}$ of $\mathfrak a^{1,0}$ and extend it to a basis $\pg{Z_1,\dots,Z_n}$ of $\mathfrak g^{1,0}$.
    In the dual basis $\pg{\alpha^j}_{j=1}^n$, we get that $d\psi\in\mathcal{I}(\alpha^1,\alpha^{\bar 1})$, for all $\psi\in\mathfrak g_\C^*$.
    So we have
    $$\Omega=\sigma^{n-2}+\alpha^1\wedge\eta+\alb1\wedge\bar\eta+i\al11\wedge\theta, $$
for some $\eta\in\Lambda^{n-3,n-2}_\mathfrak a$, $\theta\in\Lambda^{n-3,n-3}_\mathfrak a$, and 
\begin{equation}\label{domega}
d\Omega=d\sigma^{n-2}-\alpha^1\wedge d\eta-\alb1\wedge d\bar\eta,
\end{equation}
where $\alpha^1\wedge d\eta\in\mathcal{I}\pt{\al11}$. 
Consider $\omega=i\al11\wedge\sigma$, corresponding to a unitary metric $g$ on $\mathfrak g$.
Its $(n-2)$ power is $\omega^{n-2}=\sigma^{n-2}+i\al11\wedge\sigma^{n-3}$, so $d\omega^{n-2}=d\sigma^{n-2}$.
Now, since $\Omega$ is closed, using \eqref{domega} we find $0=\rest{\pt{\iota_{Z_1}d\Omega}}{\mathfrak a}=\rest{\pt{\iota_{Z_1}d\sigma^{n-2}}}{\mathfrak a}$, and %the first line in \eqref{contraction} yields
\begin{equation*}
\begin{aligned}
\iota_{Z_1}d\sigma^{n-2}&=\iota_{Z_1}d\omega^{n-2}=\iota_{Z_1}\pt{(n-2)\,\omega^{n-3}\wedge d\omega}\\
&=(n-2)\pt{(n-3)\,\omega^{n-4}\wedge\iota_{Z_1}\omega\wedge d\omega+\omega^{n-3}\wedge \iota_{Z_1}d\omega}.
\end{aligned}\end{equation*}
Notice that, by construction, $\mathfrak a$ is orthogonal to $Z_1$ with respect to $g$, so in particular $\omega^{n-4}\wedge\iota_{Z_1}\omega\wedge d\omega$ vanishes when evaluated on $\mathfrak a$.
This gives
\begin{equation*}
\begin{aligned}
0=\rest{\pt{\iota_{Z_1}d\sigma^{n-2}}}{\mathfrak a}&=(n-2)\rest{\pt{(n-3)\,\omega^{n-4}\wedge\iota_{Z_1}\omega\wedge d\omega+\omega^{n-3}\wedge \iota_{Z_1}d\omega}}{\mathfrak a}\\
&=(n-2)\rest{\pt{\omega^{n-3}\wedge \iota_{Z_1}d\omega}}{\mathfrak a}\\
&=\sigma^{n-3}\wedge\rest{\pt{ \iota_{Z_1}d\omega}}{\mathfrak a},
\end{aligned}\end{equation*}
where the last equality holds true because
\begin{equation*}
    \omega^{n-3}\wedge \iota_{Z_1}d\omega=\pt{\sigma^{n-3}+(n-4)\,i\al11\wedge\sigma^{n-4}}\wedge\pt{\rest{\iota_{Z_1}d\omega}{\mathfrak a}+\al{}1\wedge\iota_{\bar{Z}_1}\iota_{Z_1}d\omega}.
\end{equation*}
By injectivity of the Lefschetz operator we can then conclude that $\rest{\pt{\iota_{Z_1}d\omega}}{\mathfrak a}=0$, and similarly one can prove that $\rest{\pt{\iota_{{\bar Z}_1}d\omega}}{\mathfrak a}=0$.
In terms of the Chern torsion $T$, this can be read as $\iota_{{ Z}_1}T=\iota_{{\bar Z}_1}T=0$.
To conclude, it is sufficient to retrace the computations of the components of $T$ with respect to the unitary basis $\pg{Z_1,\dots,Z_n}$, as in  \cite[eq. (37)]{GuoZheng}, to find that the complex structure equations in this basis are as in \eqref{cseabid}.
\end{proof}

\begin{remark}\label{abidk} 
    The Lie algebra  $(\mathfrak g, J)$ with complex structure equations \eqref{cseabid}, such that  
    \begin{equation*}   
    \begin{aligned}
&\lambda_j= v_j=0,&&\quad\text{for all } j=2,\dots,l,\\
        &\lambda_j\neq0,&&\quad\text{for all } j=l+1,\dots,n,
    \end{aligned}\end{equation*}
    with $l\ge2$, is K\"ahler. 
    Indeed, we already know that this is true if $v_j=0$, for all $j=2,\dots,n$ \cite[Prop. $3(iii)$]{GuoZheng}.
    If this is not the case, define 
    \begin{equation*}
        w_j=
            0,                      \quad    j=1,\dots,l,\quad
        w_j=
            \frac{v_j}{\lambda_j},          \quad j=l+1,\dots,n.
        % \end{cases}
    \end{equation*}
    Unless $\mathfrak g$ is  abelian, the $w_j$ are not all zero, and the $(1,1)$-form
    \begin{equation*}
    i\sum_{j=1}^n\pt{2\abs{w_j}^2\alpha^{1\bar 1}+\alpha^{j\bar j}-w_j\,\alpha^{1\bar j}-\overline{w_j}\,\alpha^{j\bar 1}}
    \end{equation*}
    is positive definite and closed.
\end{remark}

\begin{theorem}\label{ThmAbIdeal}
    A unimodular Lie algebra of complex dimension $n\geq4$, admitting an abelian, $J$-invariant ideal $\mathfrak a$ of real codimension $2$, cannot be $(n-2)$-K\"ahler, unless K\"ahler.
\end{theorem}

\begin{proof}
    In light of Proposition \ref{abidn-2k} and Remark \ref{abidk}, the only case to consider is that with structure equations \eqref{cseabid}, with  $\lambda_k=0$, $v_k\neq0$, for some $k=2,\dots,n$.
    Up to a rescaling and reordering of the basis, we can assume
        \begin{equation*}%\begin{cases}
    d\alpha^1=\dots=d\alpha^{l-1}=0, \quad
    d\alpha^l=i\al11.
    \end{equation*}
    If $l>2$, the $3$-form $\beta=2i\,\alpha^{2\bar 2l}$, with $d\beta=\aldue{12}{1}{2}$, gives an obstruction to the existence of $(n-2)$-K\"ahler form, by Proposition \ref{obstructpK}.
    On the other hand, if $l=2$,
    \begin{equation*}
        \pt{d\aldue323}^{2,2}=\del\alpha^{3\bar2\bar3}=i\aldue{13}13+\bar v_3\aldue{13}12.
    \end{equation*}
    Now, if $v_3=0$, we conclude as above.
    If instead $v_3\neq0$, we can assume $\lambda_3\neq0$ as well, otherwise, up to a change of basis, we fall back to the case $l>2$.
    The obstruction is then given by $\beta=\aldue323+\displaystyle\frac{\bar v_3}{\bar\lambda_3}\aldue312$, for  $\del\beta=i\aldue{13}13$.
\end{proof}

In light of Theorem \ref{theoremn-2}, this is an obstruction to the existence of \hr balanced structures, proving item \ref{abideal} in Theorem \ref{ThmIntro}.

\section{A non-compact example} \label{section.noncompact}

In this section we will show the existence of a Hodge-Riemann balanced structure on a non-compact {complex manifold of complex dimension $4$}, obtained as the product of the Iwasawa manifold by $\C$.

% Consider, namely the complex Lie group of dimension 3
The Iwasawa manifold is the  complex parallelizable manifold ${\mathcal{I}}=\Gamma\backslash H$ obtained as the
quotient of the complex Heisenberg group $H$  defined as
\begin{equation*}
H\coloneqq \text{Heis}(3,\C)=\pg{
\begin{pmatrix}
    1   &   z_1 &   z_3 \\
    0   &   1   &   z_2 \\
    0   &   0   &   1
\end{pmatrix}, z_1,z_2,z_3\in\C
},
\end{equation*}
by the lattice $\Gamma\coloneqq\text{Heis}(3,\Z\pq i)$ of matrices in $\text{Heis}(3,\C)$ with Gauss integer coefficients.
This manifold is of particular interest as it  is one of the first examples of compact {complex} manifolds that are balanced and non-K\"ahler.
In fact, any invariant metric {compatible with the natural bi-invariant complex structure on $\mathcal I$} is balanced.
On $\mathcal I$ we can fix {complex} coordinates $\pg{z_1,z_2,z_3}$, such that a global holomorphic frame is given by 
\begin{equation*}
    \vartheta_1\coloneqq \frac{\partial}{\partial z_1},\quad
    \vartheta_2\coloneqq \frac{\partial}{\partial z_2}+z_1\frac{\partial}{\partial z_3},\quad
    \vartheta_3\coloneqq \frac{\partial}{\partial z_3},
\end{equation*}
with dual co-frame
\begin{equation*}
    \varphi^1\coloneqq dz_1,\quad
    \varphi^2\coloneqq dz_2,\quad
    \varphi^3\coloneqq dz_3-z_1dz_2.
\end{equation*}

{Let $u$  be the complex coordinate on $\C$}. We  study the product $M={{\mathcal I}\times \C}$, with holomorphic frame $\pg{\vartheta_1,\vartheta_2,\vartheta_3,\vartheta_4\coloneqq\frac{\partial}{\partial u}}$, and co-frame $\pg{\varphi^1,\varphi^2,\varphi^3,\varphi^4\coloneqq du}$ such that
\begin{equation*}
    d\varphi^{j}=0,\,j=1,2,4,\quad d\varphi^{3}=-\varphi^{12}.
\end{equation*}
By \cite[Thm. 4.1]{AlessandriniLieGroups}, {the complex manifold} $M$ admits $2$-K\"ahler structures, coming from any balanced structure on $\mathcal I$.
On the other hand, $M$ is not K\"ahler as, if it did admit a K\"ahler metric, the restriction of the latter to $\mathcal{I}$ would be a K\"ahler metric, giving a contradiction.

We will now give a different construction of a $2$-K\"ahler form on $M$, that is 
\begin{equation*}
\begin{aligned}    
    \Omega=&e^{\abs{u}^2}\varphi^{12\bar1\bar2}+\varphi^{13\bar1\bar3}+\pt{1+\abs{u}^2e^{\abs{u}^2}}\varphi^{14\bar1\bar4}+\varphi^{23\bar2\bar3}+\pt{1+\abs{u}^2e^{2\abs{u}^2}}\varphi^{24\bar2\bar4}\\
    &+e^{\abs{u}^2}\pt{1+\abs{u}^2}\varphi^{34\bar3\bar4}+u\,e^{\abs{u}^2}\varphi^{12\bar3\bar4}+\bar{u}\,e^{\abs{u}^2}\varphi^{34\bar1\bar2}.
\end{aligned}
\end{equation*}
For the sake of notation, we will denote $U=\abs{u}^2$.
One can easily see that $\Omega$ is closed and in fact not only $\Omega$ is transverse, but it is positive definite at every point.
To show the latter, we use Remark \ref{rmkPosDef}.
Consider the following $(2,0)$-forms on $M$
\begin{equation*}
    \Psi^1=e^U\pt{\varphi^{12}-\bar u\,\varphi^{34}},\,
    \Psi^2=\varphi^{13},\,
    \Psi^3=\varphi^{14},\,
    \Psi^4=\varphi^{23},\,
    \Psi^5=\varphi^{24} ,\,
    \Psi^6=\varphi^{34}.
\end{equation*}
At every point $x\in M$, these generate the space of $(2,0)$-forms on $T_xM$.
With respect to this basis, we see that $Q=Q_{\Omega}^{2,0}$ has a diagonal expression, as $Q(\Psi^j,\Psi^k)=0$, for $j\neq k$, and
\begin{equation*}
\begin{aligned}
        &Q(\Psi^1,\Psi^1)=e^{3U}     ,\\
    &Q(\Psi^2,\Psi^2)=Q(\Psi^4,\Psi^4)=1+e^{2U}U     ,\\
    &Q(\Psi^3,\Psi^3)=Q(\Psi^5,\Psi^5)=1     ,\\
    % &Q(\Psi^4,\Psi^4)= 1+e^{2U}U    ,\\
   % & Q(\Psi^5,\Psi^5)=1     ,\\
   & Q(\Psi^6,\Psi^6)=e^U     .
\end{aligned}
\end{equation*}
By Remark \ref{rmkPosDef}, this means that $\Omega$ is positive definite, as wanted.
Here, the reference volume form is $\varphi^{1234\bar1\bar2\bar3\bar4}$, and the Hodge star operator in the definition of $Q$ is intended to be with respect to the metric $ i\pt{\varphi^{1\bar1}+\varphi^{2\bar2}+\varphi^{3\bar3}+\varphi^{4\bar4}}$.

\begin{lemma}\label{exHRform}
 $\Omega$ is (pointwise) a Hodge-Riemann form with respect to 
\begin{equation*}
    \omega_0\coloneqq i\pt{\varphi^{1\bar1}+\varphi^{2\bar2}+e^{-U}\varphi^{3\bar3}+\varphi^{4\bar4}},
\end{equation*}
for all degrees $(p,q)$, with $p+q=2$.
\end{lemma}

\begin{proof}
We only need to prove that $Q'=Q^{1,1}_\Omega$ is positive definite on 
\begin{equation*}
    P^{1,1}\coloneqq\pg{
    \alpha\in\Lambda^{1,1}\colon \alpha\wedge\omega_0\wedge\Omega=0
    }.
\end{equation*}
    A simple computation gives
    \begin{equation}\label{exom0wOm}
        \begin{aligned}
            \omega_0\wedge\Omega=\,&
        3i\,\varphi^{123\bar1\bar2\bar3}
        +i\,\pt{2U\,e^{2U}+e^U+2}\varphi^{124\bar1\bar2\bar4}  \\
        &\,+i\,\pt{(2U+1)e^U+1+e^{-U}}
        \pt{\varphi^{134\bar1\bar3\bar4}+\varphi^{234\bar2\bar3\bar4}}        ,
        \end{aligned}        
    \end{equation}
    so primitive forms are in coordinates $\alpha=\sum_{j,k=1}^4 a_{j\bar k}\,\varphi^{j\bar k}$ such that 
    \begin{equation*}
        3\,a_{4\bar4}=-\pt{(2U+1)e^U+1+e^{-U}}
        \pt{a_{1\bar1}+a_{2\bar2}}-\pt{2U\,e^{2U}+e^U+2}a_{3\bar3}.
    \end{equation*}
    Then, $ P^{1,1}$ is generated by
    \begin{equation*}
    \begin{aligned}
        &\Xi_1\coloneqq3\,\varphi^{1\bar1}-\pt{(2U+1)e^U+1+e^{-U}}\varphi^{4\bar4},   &&\Xi_6\coloneqq\varphi^{1\bar4},    &&\Xi_{11}\coloneqq\varphi^{2\bar4},    \\
        &\Xi_2\coloneqq3\,\varphi^{2\bar2}-\pt{(2U+1)e^U+1+e^{-U}}\varphi^{4\bar4},   &&\Xi_7\coloneqq\varphi^{2\bar1},    &&\Xi_{12}\coloneqq\varphi^{3\bar4},    \\
        &\Xi_3\coloneqq 3\,e^{-U}\varphi^{3\bar3}-\pt{2U\,e^{U}+1+2e^{-U}}\varphi^{4\bar4},     &&\Xi_8\coloneqq\varphi^{2\bar3}-\bar u\,e^U\varphi^{4\bar1},    &&\Xi_{13}\coloneqq\varphi^{4\bar1},    \\
        &\Xi_4\coloneqq\varphi^{1\bar2},                                            &&\Xi_9\coloneqq\varphi^{3\bar2}-u\,e^U\varphi^{1\bar4},    &&\Xi_{14}\coloneqq\varphi^{4\bar2},    \\
        &\Xi_5\coloneqq\varphi^{1\bar3}+\bar u\,e^U\varphi^{4\bar2},                &&\Xi_{10}\coloneqq\varphi^{3\bar1}+u\,e^U\varphi^{2\bar4},    &&\Xi_{15}\coloneqq\varphi^{4\bar3}.
    \end{aligned}\end{equation*}
    One can easily see that the matrix of the bilinear form $Q'$ associated to this basis is diagonal, with positive diagonal elements, apart from the upper left $3\times3$ block, that equals $3\,B$, with 
    \begin{equation*}
       B\coloneqq \begin{pmatrix}
            2\pt{e^U(2U+1)+1+e^{-U}}    &   {e^U(U-1)+2+2e^{-U}}         &   {e^U(U+1)+2} \\
            {e^U(U-1)+2+2e^{-U}}        &   2\pt{e^U(2U+1)+1+e^{-U}}     &   {e^U(U+1)+2} \\
            e^U(U+1)+2                  &   e^U(U+1)+2                   &   2\pt{2Ue^U+1+2\,e^{-U}}
        \end{pmatrix}.
    \end{equation*}
    By Sylvester's criterion, it is clear that the matrix is positive definite, because 
    \begin{equation*}
        \det B_{3,3}=3(U+1)\pt{e^{2U}(5U+1)+4e^U+4},
    \end{equation*}
    where $B_{3,3}$ is the principal block of $B$ obtained removing the third line and the third column, and
    \begin{equation*}\begin{aligned}
       \det B&= 6\,e^{-U}(U+1)\pt{e^{4U}(9U^2-1)+e^{3U}(9U-3)+2\,e^{2U}(9U+1)+12\,e^U+8}    \\
        &=6\,e^{-U}(U+1)\pt{e^{2U}(3U+1)+3\,e^U+2}\pt{e^{2U}(3U-1)+4},
    \end{aligned}
    \end{equation*}
    positive if and only if $f(U)=e^{2U}(3U-1)+4$ is positive, for all $U\in\R_{\ge0}$.
    This is a straightforward computation, as $f(0)=3>0$, and 
    \begin{equation*}
        f'(U)=e^{2U}(6U+1)>0,
    \end{equation*}
    for all $U\ge0$.
\end{proof}

Being $\Omega$ positive definite and $\omega_0$ strongly positive, $\omega_0\wedge\Omega$ is strongly positive, as one can also check by the explicit formula \eqref{exom0wOm}, and hence defines a Hermitian metric with fundamental form $F$ defined by 
\begin{equation*}
    \frac{F^3}{3!}=\omega_0\wedge\Omega.
\end{equation*}

\begin{theorem}
    The triple $(F,\omega_0,\Omega)$ is a Hodge-Riemann balanced structure on {the complex manifold $M= {\mathcal I} \times \C$}.
\end{theorem}

\begin{proof}
    We already saw that $\Omega$ is closed, so   by construction of $F$, and    in light of Lemma \ref{exHRform}, we only have to show that $F$ is balanced, or equivalently that
    \begin{equation}\label{exeqclosed}
        0=\del\,\pt{\omega_0\wedge\Omega}=\pt{\del\omega_0}\wedge\Omega.
    \end{equation}
   It is evident that $\varphi^{ j\bar j}$ is closed, for $j=1,2,4$, and 
    \begin{equation*}
        \partial\pt{e^{-U}\varphi^{3\bar3}} =e^{-U}\pt{\bar u\,\varphi^{34\bar3}-\varphi^{12\bar3}},
    \end{equation*}
    so \eqref{exeqclosed} follows readily, completing the proof.
\end{proof}

 {
\begin{remark}
    This example suggests broader existence phenomena of non-invariant Hodge-Riemann balanced structure in the setting of  non-compact locally homogeneous manifolds.
\end{remark}
}

\bigskip

{\bf Acknowledgements.} 
The authors would like to thank Mario Garcia-Fernandez and Gueo Grantcharov for valuable discussions regarding the subject, and Riccardo Piovani for useful comments.
The authors would also like to thank the anonymous referees for useful comments.
The second author is also very grateful to Adela Latorre for insightful conversations, and the cheerful hospitality at Departamento de Matemática Aplicada, at Universidad Politécnica de Madrid.
The authors are partially supported by Project PRIN 2022 \lq \lq Geometry and Holomorphic Dynamics” and by GNSAGA (Indam). Anna Fino   is also supported  by a grant from the Simons Foundation (\#944448). 
Asia Mainenti is partly supported by the PNRR-III-C9-2023-I8 grant CF 149/31.07.2023 {\em Conformal Aspects of Geometry and Dynamics}.

\end{document}